\documentclass[12pt,a4paper]{amsart}
\usepackage{a4wide}
\usepackage{amsmath}
\usepackage{amsthm}
\usepackage{amssymb}
\usepackage{amsthm}

\usepackage{graphicx}
\usepackage{color}
\usepackage{caption}
\usepackage{subcaption}
\usepackage{url}

\newtheorem{proposition}{Proposition}
\newtheorem{lemma}[proposition]{Lemma}

\newtheorem{theorem}[proposition]{Theorem}

\theoremstyle{remark}

\allowdisplaybreaks[3]

\title[On a conjecture of Faulhuber and Steinerberger]{On a conjecture of Faulhuber and Steinerberger on the logarithmic derivative of $\vartheta_4$}
\author{Anne-Maria Ernvall-Hyt\"onen and Esa V. Vesalainen}
\address{Matematik och Statistik, {\AA}bo Akademi University, Domkyrkotorget 1, 20500 {\AA}bo, Finland}
\thanks{This work was supported by the Academy of Finland project 303820, and E.~V.~V. was supported by the Magnus Ehrnrooth Foundation.}

\begin{document}
\maketitle

\begin{abstract} Faulhuber and Steinerberger conjectured that the logarithmic derivative of $\vartheta_4$ has the property that $y^2\,\vartheta_4'(y)/\vartheta_4(y)$ is strictly decreasing and strictly convex. In this small note, we prove this conjecture. 
\end{abstract}

\section{Introduction}

The Jacobi $\vartheta$-functions are a classical topic of perennial interest. They appear in many fields of pure and applied mathematics.
Analytic properties and the behavior of these functions is crucial for applications. These properties have been studied, for instance, in \cite{coeffeycsordas, dixitroyzaharescu, Ernvall-Hytonen--Vesalainen, Faulhuber, Faulhuber2017, faulhubersteinerberger, Montgomery, schiefermayr,Solynin}.

In the following we are interested in the classical Jacobi $\vartheta$-function $\vartheta_4$: we set
\[
\vartheta_4(y)=\sum_{k=-\infty}^{\infty}(-1)^k\,e^{-\pi k^2y}
=\prod_{n=1}^\infty\left(1-e^{-2n\pi y}\right)\left(1-e^{-(2n-1)\pi y}\right)^2,
\]
for $y\in\left]0,\infty\right[$. Of course, $\vartheta_4$ is usually defined as a modular form in the upper complex half-plane, but as we are only interested in the values on the positive imaginary axis, we employ the common and very convenient abuse of notation of rotating the positive imaginary axis to the positive real axis.


This small note concentrates on proving the following theorem which was conjectured by Faulhuber and Steinerberger in \cite{faulhubersteinerberger}.
\begin{theorem}
The expression $y^2\,\vartheta_4'(y)/\vartheta_4(y)$ is strictly convex and strictly decreasing as a function of $y\in\left]0,\infty\right[$.
\end{theorem}

The proof will be structured as follows: We prove the convexity in two parts, for small and for large values of $y$ separately. After this, it is very simple and straightforward to prove that the function is decreasing.

By looking at the details of the proof, it is clear that the exponent $2$ of $y$ is not the best possible. However, by looking at graphs of the function for different values of the exponent, it immediately becomes clear that the exponent cannot be improved very much. For example, when the exponent is $2.1$, the function is not convex everywhere.

\section{Results and proofs}
We study the function
\[
f(y)=\frac{y^2\,\vartheta_4'(y)}{\vartheta_4(y)},
\]
defined for all $y\in\left]0,\infty\right[$.

\begin{theorem}
The function $f(y)=y^2\,\vartheta_4'(y)/\vartheta_4(y)$ is strictly convex for $y\in\left[1,\infty\right[$.
\end{theorem}

\begin{proof}
Notice first that
\begin{multline*}
f(y)=y^2\,\frac{\mathrm d}{\mathrm dy}\,\log\vartheta_4(y)=y^2\,\frac{\mathrm d}{\mathrm dy}\,\log \prod_{n=1}^{\infty}\left(1-e^{-2n\pi y}\right)\left(1-e^{-(2n-1)\pi y}\right)^2\\
=y^2\,\frac{\mathrm d}{\mathrm dy}\sum_{n=1}^{\infty}\left(\log \left(1-e^{-2n\pi y}\right)+2\log \left(1-e^{-(2n-1)\pi y}\right)\right)\\
=y^2\sum_{n=1}^{\infty}\left(\frac{2n\pi\,e^{-2n\pi y}}{1-e^{-2n\pi y}}+2\,\frac{(2n-1)\,\pi\,e^{-(2n-1)\pi y}}{1-e^{-(2n-1)\pi y}}\right)\\
=2\,y^2\sum_{n=1}^{\infty}\left(\frac{n\pi}{e^{2n\pi y}-1}+\frac{(2n-1)\,\pi}{e^{(2n-1)\pi y}-1}\right).
\end{multline*}
Let us now differentiate:
\begin{multline*}
f'(y)=4y\sum_{n=1}^{\infty}\left(\frac{n\pi}{e^{2n\pi y}-1}+\frac{(2n-1)\,\pi}{e^{(2n-1)\pi y}-1}\right)\\
-2\,y^2\sum_{n=1}^{\infty}\left(\frac{2\,n^2\,\pi^2\,e^{2n\pi y}}{(e^{2n\pi y}-1)^2}+\frac{(2n-1)^2\,\pi^2\,e^{(2n-1)\pi y}}{(e^{(2n-1)\pi y}-1)^2}\right).
\end{multline*}
Let us differentiate again:
\begin{multline*}
f''(y)=4\sum_{n=1}^{\infty}\left(\frac{n\pi}{e^{2n\pi y}-1}+\frac{(2n-1)\,\pi}{e^{(2n-1)\pi y}-1}\right)\\
-8y\sum_{n=1}^{\infty}\left(\frac{2\,n^2\,\pi^2\,e^{2n\pi y}}{(e^{2n\pi y}-1)^2}+\frac{(2n-1)^2\,\pi^2\,e^{(2n-1)\pi y}}{(e^{(2n-1)\pi y}-1)^2}\right)\\
-2\,y^2\sum_{n=1}^{\infty}\left(\frac{4\,n^3\,\pi^3\,e^{2n\pi y}}{(e^{2n\pi y}-1)^2}+\frac{(2n-1)^3\,\pi^3\,e^{(2n-1)\pi y}}{(e^{(2n-1)\pi y}-1)^2}\right)\\
+2\,y^2\sum_{n=1}^{\infty}\left(\frac{2\cdot 4\,n^3\,\pi^3\,e^{4n\pi y}}{(e^{2n\pi y}-1)^3}+\frac{2\,(2n-1)^3\,\pi^3\,e^{2(2n-1)\pi y}}{(e^{(2n-1)\pi y}-1)^3}\right).
\end{multline*}
First we can simplify by combining the last two rows by
\[
\frac{2\cdot 4\,n^3\,\pi^3\,e^{4n\pi y}}{(e^{2n\pi y}-1)^3}-\frac{4\,n^3\,\pi^3\,e^{2n\pi y}}{(e^{2n\pi y}-1)^2}=\frac{4\,n^3\,\pi^3\,e^{2n\pi y}}{(e^{2n\pi y}-1)^3}\left(e^{2n\pi y}+1\right)
\]
and
\[
\frac{2\,(2n-1)^3\,\pi^3\,e^{2(2n-1)\pi y}}{(e^{(2n-1)\pi y}-1)^3}-\frac{(2n-1)^3\,\pi^3\,e^{(2n-1)\pi y}}{(e^{(2n-1)\pi y}-1)^2}=\frac{(2n-1)^3\,\pi^3\,e^{(2n-1)\pi y}}{(e^{(2n-1)\pi y}-1)^3}\left(e^{(2n-1)\pi y}+1\right).
\]
The second derivative can be rewritten as
\begin{multline*}
f''(y)=\sum_{n=1}^{\infty}\left(4\,\frac{n\pi}{e^{2n\pi y}-1}-8y\frac{2\,n^2\,\pi^2\,e^{2n\pi y}}{(e^{2n\pi y}-1)^2}+2\,y^2\,\frac{4\,n^3\,\pi^3\,e^{2n\pi y}}{(e^{2n\pi y}-1)^3}\left(e^{2n\pi y}+1\right)\right) \\
+\sum_{n=1}^{\infty}\left(4\,\frac{(2n-1)\,\pi}{e^{(2n-1)\pi y}-1}-8y\,\frac{(2n-1)^2\,\pi^2\,e^{(2n-1)\pi y}}{(e^{(2n-1)\pi y}-1)^2}\right.\\
\left.+2\,y^2\,\frac{(2n-1)^3\,\pi^3\,e^{(2n-1)\pi y}}{(e^{(2n-1)\pi y}-1)^3}\left(e^{(2n-1)\pi y}+1\right)\right).
\end{multline*}
Let us now look at the terms in the sums, starting with the first sum:
\begin{multline*}
4\,\frac{n\pi}{e^{2n\pi y}-1}-8y\,\frac{2\,n^2\,\pi^2\,e^{2n\pi y}}{(e^{2n\pi y}-1)^2}+2\,y^2\,\frac{4\,n^3\,\pi^3\,e^{2n\pi y}}{(e^{2n\pi y}-1)^3}\left(e^{2n\pi y}+1\right)\\
>-8y\,\frac{2\,n^2\,\pi^2\,e^{2n\pi y}}{(e^{2n\pi y}-1)^2}
+2\,y^2\,\frac{4\,n^3\,\pi^3\,e^{2n\pi y}}{(e^{2n\pi y}-1)^3}\left(e^{2n\pi y}+1\right)\\
=\frac{8\,n^2\,\pi^2\,y\,e^{2n\pi y}}{(e^{2n\pi y}-1)^3}\left(n\pi y\left(e^{2n\pi y}+1\right)-2\left(e^{2n\pi y}-1\right)\right).
\end{multline*}
The first factor is positive, and the second factor is certainly positive for all positive $n\in\mathbb Z_+$ when $y\geqslant2/\pi$.
Let us now move to the other sum. Let us treat the case $n=1$ separately:
\begin{multline*}
4\,\frac{\pi}{e^{\pi y}-1}-8y\,\frac{\pi^2\,e^{\pi y}}{(e^{\pi y}-1)^2}+2\,y^2\,\frac{\pi^3\,e^{\pi y}}{(e^{\pi y}-1)^3}\left(e^{\pi y}+1\right)\\
=\frac{2\pi}{(e^{\pi y}-1)^3}\left(2\left(e^{\pi y}-1\right)^2-4y\pi\, e^{\pi y}\left(e^{\pi y}-1\right)+\pi^2\,y^2\,e^{\pi y}\left(e^{\pi y}+1\right)\right).
\end{multline*}
Let us now show that this expression is positive. Define for $y\in\mathbb R_+$
\[
g(y)=2\left(e^{\pi y}-1\right)^2-4y\pi\, e^{\pi y}\left(e^{\pi y}-1\right)+\pi^2\,y^2\,e^{\pi y}\left(e^{\pi y}+1\right).
\]
We have
\begin{multline*}
g''(y)=2 \,e^{\pi y}\, \pi^2+2\, e^{2 \pi y}\, \pi^2+4\, e^{2 \pi y} \,\pi \left(-4 \pi+2 \,\pi^2\, y\right)+2 \,e^{\pi y}\, \pi \left(4 \pi+2\, \pi^2\, y\right)\\+4 \,e^{2 \pi y} \,\pi^2\left(2-4 \pi y+\pi^2\, y^2\right)+e^{\pi y} \,\pi^2 \left(-4+4 \pi y+\pi^2\, y^2\right).
\end{multline*}
The last term is clearly positive when $y>{1}/{\pi}$. Since
\[
4\, e^{2 \pi y} \,\pi \left(-4 \pi+2 \,\pi^2 \,y\right)+4 \,e^{2 \pi y}\, \pi^2 \left(2-4 \pi y+\pi^2 \,y^2\right)=4\,e^{2\pi y}\,\pi\left(\pi^3\,y^2-2\,\pi^2\,y-2\pi\right)>0,
\]
when $y> {1}/{\pi}+{\sqrt{3}}/{\pi}$, the expression $g''(y)>0$ when $y> {1}/{\pi}+{\sqrt{3}}/{\pi}$. Furthermore, since
\[
g'(1)\approx 3584.5,
\]
the first derivative is also positive. It thus suffices to compute $g(1)$:
\[
g(1)\approx 55.5>0.
\]
Let us now treat the terms with $n>1$:
\begin{multline*}
4\,\frac{(2n-1)\,\pi}{e^{(2n-1)\pi y}-1}-8y\,\frac{(2n-1)^2\,\pi^2\,e^{(2n-1)\pi y}}{(e^{(2n-1)\pi y}-1)^2}+2\,y^2\,\frac{(2n-1)^3\,\pi^3\,e^{(2n-1)\pi y}}{(e^{(2n-1)\pi y}-1)^3}\left(e^{(2n-1)\pi y}+1\right)\\
>-8y\,\frac{(2n-1)^2\,\pi^2\,e^{(2n-1)\pi y}}{(e^{(2n-1)\pi y}-1)^2}
+2\,y^2\,\frac{(2n-1)^3\,\pi^3\,e^{(2n-1)\pi y}}{(e^{(2n-1)\pi y}-1)^3}\left(e^{(2n-1)\pi y}+1\right)\\
=\frac{2\,(2n-1)^2\,\pi^2\,y\,e^{(2n-1)\pi y}}{(e^{(2n-1)\pi y}-1)^3}\left((2n-1)\,\pi y \left(e^{(2n-1)\pi y}+1\right)-4\left(e^{(2n-1)\pi y}-1\right)\right)\\
>\frac{2\,(2n-1)^2\,\pi^2\,y\,e^{(2n-1)\pi y}}{(e^{(2n-1)\pi y}-1)^3}\left((2n-1)\,\pi y \,e^{(2n-1)\pi y}-4\,e^{(2n-1)\pi y}\right)>0,
\end{multline*}
when $3y\pi>4$, so certainly when $y\geqslant 1$. This completes the proof.
\end{proof}

Recall the Jacobi $\vartheta$-function $\vartheta_2$ defined for $y\in\mathbb R_+$ by
\[
\vartheta_2(y)=\sum_{n=-\infty}^{\infty}e^{-\pi y(n+1/2)^2}.
\]
We prove the following estimates for this function.

\begin{lemma}\label{sharper-theta2-approximations}
For $y\in\left[1,\infty\right[$ and $\nu\in\left\{0,1,2,3\right\}$, we have
\[0<\vartheta_{2,\nu}(y)<(-1)^\nu\,\vartheta_2^{(\nu)}(y)<\Theta_{2,\nu}(y),\]
where
\[\vartheta_{2,\nu}(y)=\frac{2\,\pi^\nu\,e^{-\pi y/4}}{4^\nu}+\frac{2\cdot9^\nu\,\pi^\nu\,e^{-9\pi y/4}}{4^\nu},\]
and
\[\Theta_{2,\nu}(y)=\frac{2\,\pi^\nu\,e^{-\pi y/4}}{4^\nu}+\frac{2\left(1+c_\nu\right)\cdot9^\nu\,\pi^\nu\,e^{-9\pi y/4}}{4^\nu},\]
where in turn
\[c_0=0.00001,\qquad
c_1=0.00003,\qquad
c_2=0.00008,\qquad\text{and}\qquad
c_3=0.0003.\]
\end{lemma}

\begin{proof}
Let us first observe that it is easy to check that the expression $t^\nu\,e^{-\pi ty/4}$ is strictly decreasing as a function of $t\in\left[24,\infty\right[$ for any fixed $y\in\left[1,\infty\right[$ and $\nu\in\left\{0,1,2,3\right\}$.
Now the key idea is to estimate
\begin{align*}
0&<(-1)^\nu\,\vartheta_2^{(\nu)}(y)-\frac{2\,\pi^\nu\,e^{-\pi y/4}}{4^\nu}-\frac{2\cdot9^\nu\,\pi^\nu\,e^{-9\pi y/4}}{4^\nu}
=\frac{2\,\pi^\nu}{4^\nu}\sum_{\substack{n\geqslant5,\\2\nmid n}}n^{2\nu}\,e^{-\pi n^2y/4}\\
&<\frac{2\,\pi^\nu}{4^\nu}\sum_{n=25}^\infty n^\nu\,e^{-\pi ny/4}
<\frac{2\cdot9^\nu\,\pi^\nu\,e^{-9\pi y/4}}{4^\nu}\cdot\frac{e^{9\pi y/4}}{9^\nu}\int\limits_{24}^\infty t^\nu\,e^{-\pi ty/4}\,\mathrm dt.
\end{align*}
The rest is simple as the last integral can be computed explicitly for each $\nu\in\left\{0,1,2,3\right\}$.
\end{proof}

\begin{theorem}
The function $f(y)=y^2\,\vartheta_4'(y)/\vartheta_4(y)$ is strictly convex for $y\in\left]0,1\right]$.
\end{theorem}

\begin{proof}
The second derivative of $f(y)$ is $h(y)/\vartheta_4^3(y)$, where
\begin{multline*}
h(y)=2\,\vartheta_4'(y)\,\vartheta_4^2(y)
+4\,y\,\vartheta_4''(y)\,\vartheta_4^2(y)
+y^2\,\vartheta_4'''(y)\,\vartheta_4^2(y)\\
-4\,y\left(\vartheta_4'(y)\right)^2\vartheta_4(y)
-3\,y^2\,\vartheta_4''(y)\,\vartheta_4'(y)\,\vartheta_4(y)
+2\,y^2\left(\vartheta_4'(y)\right)^3.
\end{multline*}
Since $\vartheta_4(y)$ is strictly positive, it is enough to prove that $h(y)>0$ for $y\in\left]0,1\right]$.

Differentiating three times the modularity relation
\[\vartheta_4(y)=y^{-1/2}\,\vartheta_2\!\left(\frac1y\right),\]
we get first
\[\vartheta_4'(y)=-\frac12\,y^{-3/2}\,\vartheta_2\!\left(\frac1y\right)
-y^{-5/2}\,\vartheta_2'\!\left(\frac1y\right),\]
then
\[\vartheta_4''(y)=\frac34\,y^{-5/2}\,\vartheta_2\!\left(\frac1y\right)
+3\,y^{-7/2}\,\vartheta_2'\!\left(\frac1y\right)
+y^{-9/2}\,\vartheta_2''\!\left(\frac1y\right),\]
and finally
\[\vartheta_4'''(y)=-\frac{15}8\,y^{-7/2}\,\vartheta_2\!\left(\frac1y\right)
-\frac{45}4\,y^{-9/2}\,\vartheta_2'\!\left(\frac1y\right)
-\frac{15}2\,y^{-11/2}\,\vartheta_2''\!\left(\frac1y\right)
-y^{-13/2}\,\vartheta_2'''\!\left(\frac1y\right).\]

Substituting these back to $h(y)$ we are left to prove that the expression
\begin{multline*}
h\!\left(\frac1y\right)
=2\,y^{9/2}\left(\vartheta_2'(y)\right)^2\vartheta_2(y)
-2\,y^{9/2}\,\vartheta_2''(y)\,\vartheta_2^2(y)
-2\,y^{11/2}\left(\vartheta_2'(y)\right)^3\\
+3\,y^{11/2}\,\vartheta_2''(y)\,\vartheta_2'(y)\,\vartheta_2(y)
-y^{11/2}\,\vartheta_2'''(y)\,\vartheta_2^2(y)
\end{multline*}
is strictly positive for $y\in\left[1,\infty\right[$.

Using Lemma \ref{sharper-theta2-approximations}, we may estimate, for $y\in\left[1,\infty\right[$,
\begin{align*}
h\!\left(\frac1y\right)
&>2\,y^{9/2}\,\vartheta_{2,1}^2(y)\,\vartheta_{2,0}(y)
-2\,y^{9/2}\,\Theta_{2,2}(y)\,\Theta_{2,0}^2(y)
+2\,y^{11/2}\,\vartheta_{2,1}^3(y)\\
&\qquad-3\,y^{11/2}\,\Theta_{2,2}(y)\,\Theta_{2,1}(y)\,\Theta_{2,0}(y)
+y^{11/2}\,\vartheta_{2,3}(y)\,\vartheta_{2,0}^2(y)\\
&=y^{9/2}\,e^{-27\pi y/4}
\left(e^{4\pi y}\left(\alpha y-\beta\right)
+e^{2\pi y}\left(-\gamma y-\delta\right)
-\varepsilon y-\zeta\right),
\end{align*}
with constants $\alpha\approx1984.32$, $\beta\approx631.718$, $\gamma\approx1985.41$, $\delta\approx631.798$, $\varepsilon\approx1.01719$ and $\zeta\approx0.0799451$. Thus, we may continue the estimations by
\begin{align*}
h\!\left(\frac1y\right)&>y^{9/2}\,e^{-27\pi y/4}
\left(e^{4\pi y}\left(1984\,y-632\right)
+e^{2\pi y}\left(-1986\,y-632\right)
-2\,y-0.08\right)\\
&\geqslant y^{9/2}\,e^{-27\pi y/4}
\left(e^{2\pi y}\left(535\cdot1984\,y-535\cdot632-1986\,y-632\right)
-2\,y-0.08\right)\\
&>y^{9/2}\,e^{-27\pi y/4}
\left(e^{2\pi y}\left(533\cdot1984\,y-534\cdot632\right)
-2\,y-0.08\right)>0.
\end{align*}
\end{proof}

\begin{theorem}
The first derivative of the function $f(y)=y^2\,\vartheta_4'(y)/\vartheta_4(y)$ is strictly decreasing for $y\in\mathbb R_+$.
\end{theorem}
\begin{proof} We have proved that the function is strictly convex, namely, the second derivative is positive. Hence, it suffices to prove that the first derivative is negative for large values of $y$. The first derivative is
\begin{multline*}
f'(y)=4y\sum_{n=1}^{\infty}\left(\frac{n\pi}{e^{2n\pi y}-1}+\frac{(2n-1)\,\pi}{e^{(2n-1)\pi y}-1}\right)\\-2\,y^2\sum_{n=1}^{\infty}\left(\frac{2\,n^2\,\pi^2\,e^{2n\pi y}}{(e^{2n\pi y}-1)^2}+\frac{(2n-1)^2\,\pi^2\,e^{(2n-1)\pi y}}{(e^{(2n-1)\pi y}-1)^2}\right).
\end{multline*}
Let us first look at the terms
\[
\frac{4n\pi y }{e^{2n\pi y}-1}-\frac{4\,n^2\,\pi^2\,y^2\,e^{2n\pi y}}{(e^{2n\pi y}-1)^2}=\frac{4\pi yn}{(e^{2n\pi y}-1)^2}\left(e^{2n\pi y}-1-yn\pi\,e^{2n\pi y}\right).
\]
The first factor is clearly positive, while the second factor is clearly negative when $y$ is sufficiently large.

Let us now move to the other terms:
\begin{multline*}
\frac{4y\pi\,(2n-1)}{e^{(2n-1)\pi y}-1}-\frac{2\,(2n-1)^2\,\pi^2\,y^2\,e^{(2n-1)\pi y}}{(e^{(2n-1)\pi y}-1)^2}\\
=\frac{2\,(2n-1)\,\pi y}{(e^{(2n-1)\pi y}-1)^2}\left(2\,(e^{(2n-1)\pi y}-1)-(2n-1)\,\pi y \,e^{(2n-1)\pi y}\right).
\end{multline*}

The first factor is clearly positive while the second one is negative for large $y$, so the product is negative. The function is thus decreasing.
\end{proof}

\end{document}